\documentclass[oneside,a4paper,leqno,10pt]{amsart}
\usepackage{graphicx}
\usepackage{amsmath}	
\usepackage{amsfonts}
\usepackage{amssymb}
\usepackage{amsthm}
\usepackage{mathrsfs}
\usepackage{CJK}
\usepackage{aliascnt}
\usepackage{latexsym,amsfonts,amssymb,amsmath,amscd}
\usepackage{graphicx,color,ifpdf}
\usepackage{fancyhdr}
\allowdisplaybreaks
\usepackage[linktocpage, colorlinks, backref = page, citecolor = blue, linkcolor = blue]{hyperref}

\newcommand {\Isom }{\mathrm{Isom}}

\newcommand {\inj}{\mathrm{inj}}
\newcommand {\vol }{\mathrm{vol}}
\newcommand {\id }{\mathrm{id}}
\newcommand {\myd }{\;\mathrm{d}}
\newcommand {\supp }{\mathrm{supp}}
\newcommand {\Hess}{\mathrm{Hess}}

\newcommand {\R}{\mathbb{R}}

\newcommand {\Ric}{\mathrm{Ric}}
\newcommand {\myarrow}[1]{\mathop{\longrightarrow}\limits^{#1}}
\newcommand {\mysup}[2]{\sup\limits_{#1}{#2}}
\newcommand{\Xint}[1]{\mathchoice
	{\XXint\displaystyle\textstyle{#1}}%
	{\XXint\textstyle\scriptstyle{#1}}%
	{\XXint\scriptstyle
		\scriptscriptstyle{#1}}%
	{\XXint\scriptscriptstyle
		\scriptscriptstyle{#1}}%
	\!\int}
\newcommand{\XXint}[3]{{
		\setbox0=\hbox{$#1{#2#3}{\int}$}
		\vcenter{\hbox{$#2#3$}}\kern-.5\wd0}}
\newcommand{\dashint}{\Xint-}

\input xy
\xyoption{all}

\newtheorem{theorem}{Theorem}[section]
\newaliascnt{lemma}{theorem}
\newtheorem{lemma}[lemma]{Lemma}
\aliascntresetthe{lemma}

\newaliascnt{corollary}{theorem}
\newtheorem{corollary}[corollary]{Corollary}
\aliascntresetthe{corollary}

\theoremstyle{definition}
\newaliascnt{definition}{theorem}
\newtheorem{definition}[definition]{Definition}
\aliascntresetthe{definition}

\newtheorem*{acknowledgements}{Acknowledgements}

\theoremstyle{remark}
\newaliascnt{remark}{theorem}
\newtheorem{remark}[remark]{Remark}
\aliascntresetthe{remark}

\title{FIBRATIONS, AND STABILITY FOR COMPACT GROUP ACTIONS ON MANIFOLDS WITH LOCAL BOUNDED RICCI COVERING GEOMETRY}
\author{Hongzhi Huang}
\email{hyyqsaax@163.com}
\address{School of Mathematical Sciences, Capital Normal University, Beijing, China}

\begin{document}

	\maketitle
	\thispagestyle{empty}

	\begin{abstract}
		In the study of collapsed manifolds with bounded sectional curvature, the following two results provide basic tools: a (singular) fibration theorem (\cite{Fu1}, \cite{CFG}), and the stability for isometric compact Lie group actions on manifolds (\cite{Pa},\cite{GK}). The main results in this paper
(partially) generalize the two results to manifolds with local bounded Ricci covering geometry.
	\end{abstract}
	
	\setcounter{section}{-1}
	
	\section{Introductions}

	In the study of collapsed manifolds with bounded sectional curvature ($|\sec|\le1$), the following two results provide basic tools: a (singular) fibration theorem (\cite{Fu1}, \cite{CFG}), and the stability for isometric compact Lie group actions on manifolds (\cite{Pa},\cite{GK}).
	
	The main purpose of this paper is to (partially) generalize the two results to
	manifolds with local bounded Ricci covering geometry (Definition \ref{LBRCG}). Our two generalized 
	results have been used in a recent work by Rong (\cite{Ro4}).
	
	A fibration theorem in \cite{Fu1} (\cite{CFG}) says:
	
	\begin{theorem} \label{FukayaFibration}
		
		Given constants, $n, i_0>0$, there exist constant $\delta(n,i_0)>0$ (depending on $n$ and $i_0$) and $C(n)>0$, such that if a compact $n$-manifold $M$ and a compact $k$-manifold $N$, $k\le n$, satisfy
		$$|\sec_{M}|\le 1,\quad |\sec_N|\le 1, \quad \inj_N\ge i_0,\quad d_{GH}(M,N)\le\delta\le\delta(n,i_0),$$where $\inj_N$ denotes the injectivity radius of $N$ and $d_{GH}$ denotes the Gromov-Hausdorff distance,
		then there is a smooth fiber bundle map, $f: M\to N$, such that
		
		(\ref{FukayaFibration}.1) $f$ is a $\Psi(\delta|n,i_0)$-GHA (Gromov-Hausdorff approximation),
		where the function $\Psi(\delta|n,i_0)\to 0$ as $\delta\to 0$ while $n$ and $i_0$ are fixed.

		(\ref{FukayaFibration}.2) $f$ is a $\Psi(\delta|n,i_0)$-Riemannian submersion, i.e., $e^{-\Psi(\delta|n,i_0)}\le\frac{|\myd f_x(\xi)|}{|\xi|}\le e^{\Psi(\delta|n,i_0)}$
		where $\xi\in T_xM$ is perpendicular to the $f$-fiber at $x$.
		
		(\ref{FukayaFibration}.3) The second fundamental form of each $f$-fiber $|\mathrm{II}_{f}|\le C(n)$.
	\end{theorem}
	
	Note that (\ref{FukayaFibration}.1) and (\ref{FukayaFibration}.3) imply that the induced metric on an $f$-fiber is almost flat, and thus each $f$-fiber is diffeomorphic to an infra-nilmanifold (\cite{Gr}, \cite{Ru}). If one replaces $N$ by a compact length space $X$ and $\delta=
	\delta(X)$, then there is a singular fibration map, $f: M\to X$ (\cite{Fu2}, \cite{CFG}).
	
	For a $r$-distance ball at $x\in M$, $B_r(x)$, let $\pi: (\widetilde {B_r(x)},\tilde x)\to (B_r(x),x)$,
	denote the (incomplete) Riemannian universal cover. Put $\widetilde{\vol}(B_r(x))\triangleq \vol(B_r(\tilde x))$, called the local rewinding volume of $B_r(x)$ (\cite{Ro2}). Then $\widetilde{\vol}(B_r(x))\ge \vol(B_r(x))$
	and ``$=$'' if and only if $B_r(x)$ is simply connected. 
	\begin{definition}\label{LBRCG}
		An $n$-manifold $M$ is said to satisfy a local $(r,v)$-bound Ricci covering geometry, if there are constants, $r,  v>0$, such that
		$$\Ric_M\ge -(n-1),\quad \widetilde{\vol}(B_r(x))\ge v,\quad \forall \, x\in M.$$
	\end{definition}

	Based on Cheeger-Colding theory on manifolds of Ricci curvature bounded below (\cite{CC1}-\cite{CC3}, \cite{Ch}), around 2014 Rong proposed to investigate the class of (collapsed) manifolds with local bounded Ricci covering geometry, partially because such (collapsed) Ricci limit spaces share similar local
	geometric and topological properties (e.g. any tangent cone is a metric cone, \cite{Ro2}).
	Since then, there have been many progress, cf., \cite{CRX1}, \cite{CRX2},  \cite{HKRX}, \cite{HR}, \cite{Pa1}, \cite{Pa2}, \cite{PR}, \cite{Ro2}-\cite{Ro4}.

	Our first main result generalizes Theorem \ref{FukayaFibration} to manifolds with local bounded Ricci covering geometry.
	
	\begin{theorem}\label{Main}
		
		Given constants, $n, i_0, v>0$, there exists a constant $\delta(n,i_0,v)>0$, such that if a compact $n$-manifold $M$ and a compact $k$-manifold $N$, $k\le n$, satisfy,
		$$\Ric_{M}\ge-(n-1),\, \widetilde{\vol}(B_1(x))\ge v,\forall \, x\in M$$$$|\sec_N|\le 1,\, \inj_N\ge i_0,\, d_{GH}(M,N)\le\delta\le\delta(n,i_0,v),$$
		then there is a smooth fiber bundle map, $f: M\to N$, such that $f$ is a $\Psi(\delta|n,i_0,v)$-GHA.
	\end{theorem}
	
	A special case in Theorem \ref{Main} is when $k=n$, the volume condition can be removed since the volume convergence theorem (\cite{Co2}) implies $\vol(B_1(x))\ge v(n,i_0)$, and $f: M\to N$ is a diffeomorphism (\cite{CC2}, \cite{CJN}).

	\begin{remark} In \cite{Ro4}, it is proved that each fiber in Theorem \ref{Main} is diffeomorphic to an infra-nilmanifold; note that the induced metric on a fiber may not have a lower Ricci curvature bound, hence Theorem A in \cite{HKRX} does not apply. Conversely, a collapsed metric with Ricci curvature bounded below whose underlying `collapsing structure' is a (singular) nilpotent fibration, that is, extrinsic diameters of
		fibers are uniformly small while normal slices to fibers are not collapsed,
		necessarily satisfies certain local bounded Ricci covering geometry condition.
	\end{remark}
	
	\begin{remark} Theorem \ref{Main} does not hold if one removes the rewinding volume condition(\cite{An}). There have been generalizations of Theorem \ref{FukayaFibration} to manifolds with Ricci curvature bounded below (or in absolute value) under various additional conditions (\cite{DWY}, \cite{HKRX}, \cite{NZ}, \cite{PWY}, etc). We point it out that in any of the above work, indeed the volume condition in Theorem \ref{Main} is satisfied, and smoothing methods play an essential role, i.e.,
		finding a nearby metric with bounded sectional curvature then apply Theorem \ref{FukayaFibration}. In the contrast, our proof gives a direct construction for a bundle map; which does not use smoothing techniques, nor it relies on Theorem \ref{FukayaFibration}. A main reason is that because of the weak regularity, the present smoothing techniques do not seem to apply to our circumstance.
	\end{remark}
	
	\begin{remark}
		We conjecture that replacing $N$ in Theorem \ref{Main} by a compact length metric space $X$, a suitable singular fibration theorem similar to that in \cite{Fu2} should hold
		when strengthen the volume condition to the following condition: all points on $M$ are uniform $(r,\delta)$-local rewinding
		Reifenberg points, i.e., for all $x\in M$, $d_{GH}(B_s(\tilde x),B_s(0^n))\le s\delta(n)$ for all $0<s\le r$, where $\tilde x\in \widetilde{B_r(x)}$ and $B_s(0^n)\subset\R^n$ (\cite{HKRX}).
	\end{remark}

	To state the second main result in this paper, recall that the stability for compact Lie group $G$-actions
	in \cite{Pa} says that, if two $G$-actions on a compact manifold $M$ are $C^1$-close, then
	the two $G$-actions are conjugate via a small perturbation of the identity map. In \cite{GK},
	a geometric criterion for the $C^1$-closeness is described. For the sake of simplicity,
	we state the following often used version.
	
	\begin{theorem}\label{GKstability}
		Given $i_0>0$, there exists $\delta(i_0)>0$, such that the following holds.
		Let $M$ be a compact manifold with two metrics $g_0$ and $g_1$ such that $|\sec_{g_i}|\le1$ and $\inj_{g_0}\ge i_0$, and let $G$ be a compact Lie group
		with two effective isometric actions, $\iota_i:G\hookrightarrow\Isom(M,g_i)$, $i=0,1$, where $\Isom(M,g_i)$ denotes the isomtery group of  $(M,g_i)$. If the two actions are $\delta$-equivariant close, $\delta\le\delta(i_0)$, i.e., the identity map $\id_M:(M,g_1)\to(M,g_0)$ is a $\delta$-GHA and $d_{g_0}(\iota_0(g)x,\iota_1(g)x)\le\delta$ for all $x\in M,g\in G$, then the $\iota_0$-action and the $\iota_1$-action are conjugate by a diffeomorphism which is (small) isotopic to $\id_M$.
	\end{theorem}
	
	We generalize Theorem \ref{GKstability} to the following:
	
	\begin{theorem}\label{GroupStatbility}
		Given $n, i_0>0$, there exists $\delta(n,i_0)>0$, such that the following holds.
		Let $M$ be a compact $n$-manifold with two metrics $g_0$ and $g_1$ such that $|\sec_{g_0}|\le1$, $\inj_{g_0}\ge i_0$, $\Ric_{g_1}\ge -(n-1)$ and let $G$ be a compact Lie group effectively acting on $(M,g_0)$ and $(M,g_1)$ by isometries respectively. If the two actions are $\delta$-equivariant close, $\delta\le\delta(n,i_0)$, then the two actions are conjugate by a diffeomorphism.	
	\end{theorem}
	
	We believe that in Theorem \ref{GroupStatbility} the diffeomorphism should be small isotropic to $\id_M$. We point it out that the regularity condition in Theorem \ref{GroupStatbility} is weak; the conditions only guarantee a bi-H\"older close for distances $d_{g_i}$, while the conditions in Theorem \ref{GKstability} imply the $C^{1,\alpha}$-close for tensors $g_i$.
	
	As applications of Theorems \ref{Main} and \ref{GroupStatbility}, we give an equivariant fibration theorem, and a quantitative maximal volume rigidity result.
	
	\begin{corollary}\label{EquivFibration}
		Given $n,v>0$, a compact $k$-manifold $N$ and a closed subgroup $H\subset\Isom(N)$, $k\le n$, there exists $\delta_0$ depending on $n,v,N,H$, satisfying the follows. If $M$ is a compact $n$-manifold and $G$ is a closed subgroup of $\Isom(M)$, satisfying for any $x\in M$, $$\Ric_{M}\ge-(n-1),\quad\widetilde{\vol}(B_1(x))\ge v,\quad d_{GH}((M,G),(N,H))\le\delta\le\delta_0,$$ where $d_{GH}$ denotes the equivariant Gromov-Hausdorff distance, then there exits a smooth fiber bundle map $f:M\to N$ and a Lie group homomorphism $\varphi:G\to H$ such that for any $g\in G$, $f\circ g=\varphi(g)\circ f$.
	\end{corollary}

	\begin{corollary}\label{SphereRigidity}
		For $n>0$, there exists $\delta_0(n)>0$ such that if $M$ is a compact $n$-manifold with $\Ric_{M}\ge (n-1)$ and $\vol(\tilde M)\ge(1-\delta)\vol(S_1^n)$ for $\delta\in(0,\delta_0(n))$, where $S_1^n$ denotes the standard unit $n$-sphere, then $M$ is diffeomorphic to a spherical space form by a $\Psi(\delta|n)$-GHA.	
	\end{corollary}
	
	\begin{remark} Corollary \ref{SphereRigidity} partially recovers Theorem A in \cite{CRX1}; the volume condition of Corollary \ref{SphereRigidity} was proved in \cite{CRX1} by assuming local rewinding volume almost maximal and universal cover non-collapsed. To conclude Corollary \ref{SphereRigidity}, \cite{CRX1} used smoothing via Ricci flows, and then applied Theorem \ref{GKstability}. Here we can apply Theorem \ref{GroupStatbility} to have a direct proof through the weak regularity.
	\end{remark}
	
	We now briefly describe our approach to Theorem \ref{Main} and \ref{GroupStatbility}. In Theorem \ref{Main}, starting with a $\delta$-GHA, $h:M\to N$, locally we will employ $(\delta,k)$-splitting maps to approximate $h$ (\cite {CC1}, \cite{CJN}),
	and glue together these local $(\delta,k)$-splitting maps via the technique
	of center of mass, to form a smooth $\Psi(\delta|n)$-GHA, $f: M\to N$.

	To check that $\myd f$ is non-degenerate at every point (and thus $f$ is a fiber bundle map), the main difficulty is a lack of metric regularity of $M$. The verification is in two steps.
	
	The first step is to show each $(\delta,k)$-splitting map is non-degenerate (see Lemma \ref{NonDegenaracyofSplittingMap}). In the non-collapsed situation, i.e., $n=k$, the non-degeneracy is guaranteed by Canonical Reifenberg Theorem (\cite{CJN}, see Theorem \ref{CanonicalReifenbergThm} below). When $k<n$, a $(\delta,k)$-splitting map may be not a bundle map in general (\cite{An}).
	In our circumstance, the volume condition is indeed equivalent to that $M$ is uniformly local rewinding Reifenberg (see Lemma \ref{RegularPoints}), hence the lift of any $(\delta,k)$-splitting map to a local Riemannian universal cover extends to a $(\delta,n)$-splitting map. Consequently the $(\delta,k)$-splitting map is non-degenerate.

	The second step relies on a flexibility of $(\delta,k)$-splitting maps (a harmonic map which is $C^0$-close to a $(\delta,k)$-splitting map is also a $(\Psi(\delta|n),k)$-splitting map). To accentuate our idea, we illustrate it in the simple case, $N=\R^1$ and $f=\sum_\lambda \phi_\lambda v_\lambda$, where the sum has at most $\Lambda(n)$-many nonzero terms and $\phi_\lambda=(\phi\circ v_\lambda)(\sum_{\lambda'} \phi\circ v_{\lambda'})^{-1}$ is a partition of unity for a fixed cut-off function $\phi$. For $p\in M$, a direct computation shows $\myd f(p)=\sum_\lambda C_{\lambda,p}\myd v_\lambda(p)$, where $C_{\lambda,p}$ are constants depending on $p$ with $|C_{\lambda,p}|\le C(n)$ and $|\sum_\lambda C_{\lambda,p}-1|\le\Psi(\delta|n)$ (see Lemma \ref{Jacobian}). Note that each $v_\lambda$ approximates $h$, hence by the flexibility of $(\delta,k)$-splitting maps, $v(x)\triangleq\sum_\lambda C_{\lambda,p}v_\lambda(x)$ is a $(\Psi(\delta|n),k)$-splitting map around $p$ with $\myd f(p)=\myd v(p)$. By now the non-degeneracy
	of $\myd f$ at $p$ follows from the first step.

	In Theorem \ref{GroupStatbility}, to obtain a $G$-equivariant map $F:M\to M$, a standard procedure is to average $\id_M$ over the two compact isometric $G$-action via the center of mass technique. In Theorem \ref{GKstability}, the $C^{1,\alpha}$-regularity of $\id_M:(M,g_1)\to(M,g_0)$ guarantees $F$ is a diffeomorphism. In our weak regularity situation, we replace $\id_M$ by a diffeomorphism $f$ constructed in Theorem \ref{Main} for $N=(M,g_0)$ which approximates $\id_M$. With extra regularity on $f$ (Lemma \ref{Jacobian0}), by techniques similar to that used in the proof of Theorem \ref{Main} we show the non-degeneracy of $F$.

	We will organize the rest of the paper as follows:
	
	In Section 1, we review notions and properties that will be used through the rest of the paper.
	
	In Section 2, we will prove that any $(\delta,k)$-splitting map, under the conditions of Theorem \ref{Main}, is non-degenerate, see Lemma \ref{NonDegenaracyofSplittingMap}.
	
	In Section 3, we will complete the proof of Theorem \ref{Main} by showing that the gluing map is non-degenerate, see Lemma \ref{Jacobian}.
	
	In Section 4, we will prove Theorem \ref{GroupStatbility} and Corollary \ref{EquivFibration}, \ref{SphereRigidity}.

	\begin{acknowledgements}
		I would like to thank  my thesis advisor, Professor Xiaochun Rong, for suggesting to me the topic in
		this paper, and for his inspirational suggestions and support throughout the writing of 
		the paper.
	\end{acknowledgements}
	
	\noindent\\[4mm]
	\section{Preliminary}
	In this section we recall some notions and results that will be used through this paper.
	
	\subsection{Equivariant Gromov-Hausdorff Convergence}
	
	The references of this part are \cite{FY}, \cite{Ro1}.
	
	Given compact metric spaces $X,Y$, $d_{GH}(X,Y)\le\delta$, means that there is a $\delta$-GHA $h:X\to Y$, i.e., for any $x_1,x_2\in X$, $|d(x_1,x_2)-d(h(x_1),h(x_2))|\le\delta$ and for any $y\in Y$, there is $x\in X$ such that $d(h(x),y)\le\delta$. Suppose $\Gamma,G$ are closed subgroups of $\Isom(X)$ and $\Isom(Y)$ respectively. The equivariant Gromov-Hausdorff distance $d_{GH}((X,\Gamma),(Y,G))\le\delta$ means that there is a triple of maps $h: X\to Y$, $\varphi:\Gamma\to G$ and $\psi:G\to\Gamma$, for all $x\in X$, $\gamma\in\Gamma$, $g\in G$, such that,

	\begin{equation}\label{almostcommutes}
	h\text{ is a } \delta\text{-GHA},\quad
	d(\varphi(\gamma)h(x),h(\gamma x))\le\delta,\quad
	d(g h(x),h(\psi(g)x))\le\delta.
	\end{equation}
	
	The triple, $(h,\varphi,\psi)$, satisfying conditions (\ref{almostcommutes}) is called a $\delta$-equivariant GHA. We say that $(X_i,G_i)$ converges to $(X,G)$ in equivariant Gromov-Hausdorff sense, denoted by $(X_i,G_i)\myarrow{GH}(X,G)$, if $d_{GH}((X_i,G_i),(X,G))\le\delta_i\to0$.
	
	When $(X,p),(Y,q)$ are pointed complete metric spaces with closed subgroups $\Gamma\subset\Isom(X)$ and $G\subset\Isom(Y)$, then the above notions of equivariant convergence naturally extend to a pointed version. A triple of maps $(h,\varphi,\psi)$ is called a $\delta$-equivariant GHA from $(X,p,\Gamma)$ to $(Y,q,G)$ if $h:
	B_{\delta^{-1}}(p)\to B_{\delta^{-1}+\delta}(q)$, $h(p)=q$,
	$\varphi: \Gamma(\delta^{-1})\to G(\delta^{-1}+\delta)$,
	$\varphi(e)=e$, $\psi:G(\delta^{-1})\to \Gamma(\delta^{-1}
	+\delta)$, $\psi(e)=e$, and conditions (\ref{almostcommutes}) holds whenever $\gamma,g$ stay in the domain of $\varphi,\psi$ respectively and $x, \gamma x, \psi(g)x$ stay in the domain of $h$, where $\Gamma(R)\triangleq\{\gamma\in\Gamma|d(p,\gamma(p))\le R\}$. And $(X_i,p_i,\Gamma_i)\myarrow{GH}(X,p,\Gamma)$ means that there exists a sequence of $\delta_i$-equivariant GHA from $(X_i,p_i,\Gamma_i)$ to $(X,p,\Gamma)$.
	
	Some basic properties we need are listed below.
	\begin{lemma}\label{eqconexist}
		If $(X_i,p_i)\myarrow{GH}(X,p)$, and $\Gamma_i$ are closed subgroups of $\Isom(X_i)$, then passing to a subsequence, there is a closed subgroup $\Gamma\subset\Isom(X)$ such that, $$(X_i,p_i,\Gamma_i)\myarrow{GH}(X,p,\Gamma).$$
	\end{lemma}
	
	\begin{lemma} \label{eqconquotient}
		If $(X_i, p_i, \Gamma_i)\myarrow{GH}(X,p,\Gamma)$, then $(X_i/\Gamma_i, [p_i])\myarrow{GH}(X/\Gamma,[p])$, where $[\cdot]$ denotes the equivalence class.
	\end{lemma}
	
	We often apply equivariant convergence to universal covering spaces with fundamental group actions; let $\pi_i: (\tilde X_i,
	\tilde p_i)\to (X_i,p_i)$ be universal covers and $\Gamma_i$ be the fundamental group of $X_i$ such that $\tilde X_i/\Gamma_i=X_i$, if $(\tilde X_i,\tilde p_i)\myarrow{GH}(Y,y^*)$, then by Lemma \ref{eqconexist} and \ref{eqconquotient}, passing to a subsequence, we have the following equivariant commutative diagram:
	
	\begin{equation}\label{Equivariantdiagram}
	\xymatrix{
		(\tilde{X}_i,\tilde{p}_i,\Gamma_i) \ar[rr]^{GH}\ar[d]_{\pi_i}&&(Y,y^*,G) \ar[d]^{\pi_\infty} \\
		(X_i,p_i)\ar[rr]^{GH}&  & (Y/G,p),}
	\end{equation}
	where $\pi_\infty$ is the limit map of $\pi_i$.
	\subsection{The Canonical Reifenberg Theorem}
	
	A key tool in our proof of Theorems \ref{Main} and \ref{GroupStatbility} is the Canonical Reifenberg theorem
	in \cite{CJN}: any $(\delta,n)$-splitting map on a $2$-ball is bi-H\"older and non-degenerate
	in the $1$-ball (see Theorem \ref{CanonicalReifenbergThm}).
	
	Let $M$ be an $n$-manifold with $\Ric_{M}\ge-(n-1)\delta$ and $p\in M$. 
	\begin{definition}\label{sp}
		A $(\delta,k)$-splitting map, $u:B_r(p)\to\R^k$, means that, $u\in C^2$ satisfies, for each $\alpha,\beta=1,2,..,k$,
		
		(\ref{sp}.1) $\Delta u^\alpha=0$,
		
		(\ref{sp}.2) $\mysup{B_r(p)}{|\triangledown u^\alpha|}\le 1+\delta$,
		
		(\ref{sp}.3) $\dashint_{B_r(p)}|g(\triangledown u^\alpha,\triangledown u^\beta)-\delta^{\alpha\beta}|\le\delta$,
		
		(\ref{sp}.4) $r^2\dashint_{B_r(p)}|\Hess u^\alpha|^2\le\delta^2$.

	\end{definition}
	
	A geometric consequence of a $(\delta,k)$-splitting map is,
	
	\begin{theorem} \label{splittingthm}
		For $\delta\le\delta(n)$, the following holds. Let $(M,p)$ be an $n$-manifold with $Ric_{M}\ge-(n-1)\delta$.
		
		(\ref{splittingthm}.1) If there exists a $(\delta,k)$-splitting map $u: B_2(p)\to\R^k$, then $B_1(p)$ is $\Psi(\delta|n)$-GH close to $B_1(0^k,x^*)\subset\R^k\times X$ for some length space $X$.
		
		(\ref{splittingthm}.2) If $B_4(p)$ is $\delta$-close to $B_4(0^k,x^*)\subset\R^k\times X$, then there exists a $(\Psi(\delta|n),k)$-splitting map $u:B_2(p)\to \R^k$.
	\end{theorem}
	
	Comparing to early version of Theorem \ref{splittingthm} in \cite{CC1}, here the improvement is that the
	ratio of radii of balls is not necessarily tends to infinity (\cite{CJN}).

	In the case that $k=n$, the non-degeneracy of a $(\delta,n)$-splitting is crucial for
	us.
	
	\begin{theorem}\label{CanonicalReifenbergThm}
		There exists a $\delta(n)>0$ satisfying the follows. Let $(M,p)$ be an $n$-manifold satisfying $\Ric_{M}\ge-(n-1)\delta$ and $d_{GH}(B_4(p),B_4(0^n))\le\delta$. If $\delta\le\delta(n)$, Then any $(\delta,n)$-splitting map $u:B_2(p)\to\R^n$ is non-degenerate on $B_1(p)$.
	\end{theorem}

	\noindent\\[4mm]

	\section{Local rewinding volume and $(\delta,k)$-splitting maps}
	
	In this section, we will prove each $(\delta,k)$-splitting map on $M$ in Theorem \ref{Main} is non-degenerate (Lemma \ref{NonDegenaracyofSplittingMap}). In the next section, we will glue splitting maps together to construct
	a global map, $f: M\to N$, via the techniques of center of mass and use Lemma \ref{NonDegenaracyofSplittingMap} to verify the non-degeneracy of $\myd f$.
	
	For $M, N$ as in Theorem \ref{Main}, without loss of generality, by scaling we
	we may assume that $\Ric_{M}\ge-(n-1)\delta$, $\widetilde{\vol}(B_8(p))\ge v$, $|\sec_N|\le\delta$,
	$\inj_N\ge\delta^{-1}$, $d_{GH}(M,N)\le\delta$. We have $d_{GH}(B_8(p),B_8(0^k))\le\Psi(\delta|n)$.
	
	The following is the key lemma through our paper.

	\begin{lemma}\label{NonDegenaracyofSplittingMap}
		
		For $n,v>0$, there exists $\delta(n,v)>0$ satisfying the follows. Let $n$-manifold $(M,p)$ satisfy $$\Ric_{M}\ge-(n-1)\delta,\quad d_{GH}(B_8(p),B_8(0^k))\le\delta\le\delta(n,v),\quad \widetilde{\vol}(B_8(p))\ge v.$$  Then any $(\delta,k)$-splitting map $u:B_4(p)\to\R^k$ is non-degenerate on $B_2(p)$.
		
	\end{lemma}
	
	\begin{proof}
		Argue by a contradiction, assuming that a sequence of complete $n$-manifolds $(M_i,p_i)$ with $$\Ric_{M_i}\ge-(n-1)\delta_i\to0,\quad d_{GH}(B_8(p_i),B_8(0^k))\le\delta_i,\quad \widetilde{\vol}(B_4(p_i))\ge v,$$
		and (by Theorem \ref{splittingthm}) there exists a $(\delta_i,k)$-splitting map, $u_i:B_4(p_i)\to\R^k$, which
		is singular at $q_i\in B_2(p_i)$.
		
		Let $\pi_i: (\widetilde{B_8(p_i)},\tilde p_i)\to (B_8(p_i),p_i)$ be the universal cover, and let
		$\tilde q_i\in B_2(\tilde p_i)$ such that $q_i=\pi_i(\tilde q_i)$. We may assume that
		$(B_{1}(\tilde q_i),\tilde q_i)\myarrow{GH}(Y,y^*)$. By Lemma \ref{RegularPoints} below, the tangent cone at $y^*$ is $\R^n$. By volume convergence and almost maximal volume rigidity (Theorem 0.1 and 0.8 in \cite{Co2}, see also Theorem A.1.5 in \cite{CC2}),
		for any $r_i\to 0$,  $(B_{r_i^{-1}}(\tilde q_i,r_i^{-1}\tilde d_i),\tilde q_i)\myarrow{GH}(\R^n,0^n)$, where $\tilde d_i$ is the distance function on $\widetilde{B_8(p_i)}$.

		By Theorem \ref{splittingthm}, $u_i:B_{1}(q_i)\to B_{1+\epsilon_i}(0^k)$ is an $\epsilon_i$-GHA. Given $r_i\to 0$, observe that by a standard diagonal argument, passing to a subsequence, (and a suitable re-arrange indices), the following holds:

		(a)\, $\bar u_i\triangleq r_i^{-1}u_i:(B_{r_i^{-1}}(q_i,r_i^{-1}d_i),q_i)\to(B_{r_i^{-1}+r_i}(0^k),0^k)$ is an $r_i$-GHA,
		
		(b)\, $\bar u_i|_{B_{2}(q_i,r_i^{-1}d_i)}:B_{2}(q_i,r_i^{-1}d_i)\to\R^k$ is an $(r_i,k)$-splitting map,
		
		(c)\, there exists an $r_i$-GHA $h_i:(B_{r_i^{-1}}(\tilde q_i,r_i^{-1}\tilde d_i),\tilde q_i)\to (B_{r_i^{-1}+r_i}(0^n),0^n)$.
		
		(a)-(c) imply the following commutative diagram,
		$$
		\xymatrix{
			(B_{r_i^{-1}}(\tilde q_i,r_i^{-1}\tilde d_i),\tilde q_i) \ar[rr]^{h_i}\ar[d]_{\pi_i}&&(\R^n,0^n) \ar[d]^{\pi_{\infty}} \\
			(B_{r_i^{-1}}(q_i,r_i^{-1}d_i),q_i)\ar[rr]^{\bar u_i}&  & (\R^k,0^k),}
		$$
		where $\pi_\infty$ is a limit map of $\pi_i$, hence a submetry. Because any
		submetry from $\R^n$ to $\R^k$ is the standard projection, we may assume $\pi_\infty(y^1,...,y^n)=(y^1,...,y^k)$ where $(y^1,..,y^n)$ are the standard coordinates in $\R^n$. By Theorem \ref{splittingthm} again, for large $i$ we may assume
		a $(\Psi_i,n)$-splitting map $v_i=(v_i^1,...,v_i^n): B_2(\tilde q_i,r_i^{-1}\tilde d_i)\to\R^n$ such that $|v_i-h_i|\le\Psi_i\to0$. Hence for each $\alpha=1,..,k$, $$|(\bar u_i\circ\pi_i)^\alpha-v_i^\alpha|=|(\bar u_i\circ\pi_i)^\alpha-(\pi_\infty\circ v_i)^\alpha|\le |\bar u_i\circ\pi_i-\pi_\infty\circ h_i|+\Psi_i\le\Psi_i'\to 0.$$
		Then $((\bar u_i\circ\pi_i)^1,..,(\bar u_i\circ\pi_i)^k,v^{k+1},...,v^n):B_1(\tilde q_i,r_i^{-1}\tilde d_i)\to\R^n$ is a $(\Psi_i'',n)$-splitting map with $\Psi_i''\to0$. Now apply Theorem \ref{CanonicalReifenbergThm}, $\myd u_i$ is non-degenerate at $q_i$ which yields a contradiction.
		
	\end{proof}

	The follow observation is proved in \cite{HKRX}. For convenient of readers, we give details here.

	\begin{lemma}\label{RegularPoints}
		Suppose we have the following equivariant commutative diagram,
		$$
		\xymatrix{
			(\tilde{M}_i,\tilde{p}_i,\Gamma_i) \ar[rr]^{GH}\ar[d]_{\pi_i}&&(Y,y_0,G) \ar[d]^{\pi_\infty} \\
			(M_i,p_i)\ar[rr]^{GH}&  & (X,x_0)=(Y/G,x_0).}
		$$
		where $\Ric_{M_i}\ge-(n-1)$, $\vol(B_1(\tilde p_i))\ge v>0$. If a tangent cone at $x_0$ is $\R^k$, then any tangent cone at $y_0$ is $\R^n$.
	\end{lemma}
	\begin{proof}
		Let $\lambda_i\to\infty$ such that $(\lambda_iX,x_0)\myarrow{GH}(\R^k,0^k)$. By Theorem 5.2 in \cite{CC2}, any tangent cone on $Y$ is a metric cone. Hence passing to a subsequence, $(\lambda_iY_i,y_0)\myarrow{GH}(\R^m\times C(Z),(0^m,z^*))$ where $C(Z)$ is the metric cone over $Z$ with diameter $<\pi$. Thus by diagram (\ref{Equivariantdiagram}), passing to a subsequence, we have the following equivariant commutative diagram,
		
		\begin{align*}
		\xymatrix{
			(\lambda_iY,y_0,G) \ar[rr]^{GH}\ar[d]_{\pi_i}&&(\R^m\times C(Z),(0^m,z^*),H) \ar[d]^{\pi_\infty} \\
			(\lambda_iX,x_0)\ar[rr]^{GH}&  & (\R^k,0^k)=(\R^m\times C(Z)/H,0^k),}
		\end{align*}
		where $z^*$ is the cone point of $C(Z)$. Note that $\Isom(C(Z))=\Isom(Z)$ and $\Isom(\R^m\times C(Z))=\Isom(\R^m)\times\Isom(C(Z))$. Hence for each $h\in H$, there is a unique $(h_1,h_2)\in\Isom(\R^m)\times\Isom(C(Z))$, such that for any $(u,z)\in \R^m\times C(Z)$, $h(u,z)=(h_1(u),h_2(z))$. Note that $h_2(z^*)=z^*$ for all $h\in H$.
		
		It suffices to show $C(Z)$ is a point. If not, let $z^*\neq z_0\in C(Z)$. Observe that
		for any $h\in H$, $u\in\R^m$,
		\begin{equation}\label{Conesplitting}
		h(u,z^*)=(h_1(u),h_2(z^*))=(h_1(u),z^*)\neq (0^m,z_0).
		\end{equation}

		Hence, $\pi_\infty(0^m,z_0)\neq \pi_\infty(0^m,z^*)=0^k$.
		Consider the horizontal lifting line $\tilde \gamma$ at $(0^m,z^*)$ of the line $\gamma(t)=t\pi_\infty(0^m,z_0)$, $t\in(-\infty,+\infty)$. Since $C(Z)$ contains no lines, there is a line $u(t)$ in $\R^m$ such that $\tilde \gamma(t)=(u(t),z^*)$. And because $\pi_\infty(\tilde \gamma(1))=\gamma(1)=\pi_\infty(0^m,z_0)$, there is $h\in H$, such that
		$h(u(1),z^*)=h(\tilde \gamma(1))=(0^m,z_0)$, a contradiction to (\ref{Conesplitting}).
		
	\end{proof}

	\noindent\\[4mm]

	\section{Gluing splitting maps to a global bundle map}
	
	Let $M$ and $N$ be as in Theorem \ref{Main}, which, after a suitable scaling, satisfy
	the conditions described in the second paragraph of Section 2, and let $h:M\to N$ be a $\delta$-GHA.
	Fix an $1$-net $\{p_\lambda|\lambda=1,2,..,\Lambda\}$ on $M$ and an orthonormal coordinate on each $T_{h(p_\lambda)}N$. By Theorem \ref{splittingthm},
	we start with a family of $(\Psi(\delta|n),k)$-splitting maps, $u_\lambda: B_4(p_\lambda)\to T_{h(p_\lambda)}N$, approximating $\exp_{h(p_\lambda)}^{-1}\circ h|_{B_4(p_\lambda)}$. We glue local maps, $f_\lambda\triangleq\exp_{h(p_\lambda)}\circ u_\lambda$, via the center of mass techniques, to form a (global) map $f$. In the rest of this section,  our main effort is to
	verify the non-degeneracy of $f$.
	
	Let $\phi:[0,\infty)\to[0,1]$ be a smooth function with $\phi|_{[0,1.1]}\equiv1$, support $\supp\phi\subset[0,2]$ and $|\phi'|\in[0,10]$. Let $r_\lambda$ be the distance function to the $0^k$ on $T_{h(p_\lambda)}N$. Then $\phi_\lambda\triangleq\phi\circ r_\lambda\circ u_\lambda$ is a smooth function on $M$ with $\supp{\phi_\lambda}\subset B_{2.1}(p_\lambda)$. Define an energy function $E:M\times N\to\R$ by
	\begin{equation}\label{Energy}
	E(x,y)=\frac 12\sum_{\lambda} \phi_\lambda(x) d(f_\lambda(x),y)^2.
	\end{equation}
	
	Fixing $x\in M$, we may assume that $E(x,\cdot)$ is strictly convex on $B_1(h(x))$, and
	achieves the global minimum at a unique
	point, denoted by $cm(x)$, which is $\Psi(\delta|n)$-close to $h(x)$. Define a map, $f: M\to N$, $f(x)=cm(x)$. Note that
	$D(x)\triangleq\sum_\lambda\phi_\lambda(x)$ is not normalized to $1$, hence $\{\phi_\lambda\}$ is not a partition of unity. We purposely use $\{\phi_\lambda\}$ for a simple computation (indeed $cm(x)$ does not change when replacing $\phi_\lambda(x)$ by $\bar\phi_\lambda(x)\triangleq\phi_\lambda(x)D(x)^{-1}$ in definition (\ref{Energy})).
	
	\begin{lemma}\label{temp} Let the assumptions be as in Theorem \ref{Main}, and let
		$f: M\to N$ be defined as above. Then $\myd f$ is non-degenerate, and thus
		$f$ is a fiber bundle map.
	\end{lemma}
	
	By Lemma \ref{temp}, the proof of Theorem \ref{Main} is finished. In the proof of Lemma \ref{temp}, we need the following technical lemma.
	
	\begin{lemma}\label{Jacobian}
		Given $p\in M$, and
		a normal coordinate $y^1,...,y^k$ on $T_{f(p)}N$, for those $\lambda$ with $d(p,p_\lambda)<2.5$, there exist constants $C^\alpha_{\lambda,\beta}$, $\alpha,\beta=1,...,k$, and $(\Psi(\delta|n),k)$-splitting maps, $v_\lambda:B_{1}(p)\to T_{f(p)}N$, such that,
		
		(\ref{Jacobian}.1) $\myd f^\alpha(p)=\sum_{\lambda}C^{\alpha}_{\lambda,\beta}\myd v_\lambda^\beta(p)$.
		
		(\ref{Jacobian}.2) $|v_\lambda-\exp_{f(p)}^{-1}\circ h|\le\Psi(\delta|n)$.
		
		(\ref{Jacobian}.3) $|C^\alpha_{\lambda,\beta}-\delta^\alpha_\beta\phi_\lambda(p)D(p)^{-1}|\le\Psi(\delta|n)$.
		
		
	\end{lemma}
	We adopt Einstein summation convention in (\ref{Jacobian}.1) and below.
	
	\begin{proof} [Proof of Lemma \ref{temp} by assuming Lemma \ref{Jacobian}]\quad
		
		For $p\in M$, by Lemma \ref{Jacobian} we can define the following functions on $B_{1}(p)$:
		$$v^\alpha(x)\triangleq\sum_\lambda C^\alpha_{\lambda,\beta} v_\lambda^\beta(x),\quad\alpha=1,...,k.$$
		By conclusions (\ref{Jacobian}.2) and (\ref{Jacobian}.3), $\left|v^{\alpha}(x)-v_{\lambda}^\alpha(x)\right|\le\Psi(\delta|n)$. So $$v=(v^1,...,v^k)|_{B_{0.5}(p)}:B_{0.5}(p)\to\R^k$$ is a ($\Psi(\delta|n),k)$-splitting map. By (\ref{Jacobian}.1), $\myd f^\alpha(p)=\myd v^\alpha(p)$. Thus Lemma \ref{NonDegenaracyofSplittingMap} implies the non-degeneracy of $\myd f(p)$.
	\end{proof}
	
	\begin{proof}[Proof of Lemma \ref{Jacobian}]\quad
		
		For each $\lambda$ with $d(p,p_\lambda)<2.5$, there exists an isometric map $\omega_\lambda:T_{f(p)}N\to T_{h(p_\lambda)}N$ (respect to the standard Euclidean metrics on $T_{f(p)}N$ and $T_{h(p_\lambda)}N$) such that
		\begin{equation}\label{C^1closedness}
		|\omega_\lambda-\exp^{-1}_{h(p_\lambda)}\circ\exp_{f(p)}|_{C^1(B_{10}(0^k))}\le\Psi(\delta|n).
		\end{equation}
		Such $\omega_\lambda$ can be constructed as follows. Let $e_\alpha\in T_{f(p)}N$ such that $y^\beta(e_\alpha)=\delta_{\alpha}^\beta$ and  $\mu\triangleq\exp^{-1}_{h(p_\lambda)}\circ\exp_{f(p)}$, then $\vec v_\alpha\triangleq\mu(e_\alpha)-\mu(0^k)$, $\alpha=1,...,k$ is a $\Psi(\delta|n)$-almost orthonormal basis in $T_{h(p_\lambda)}N$. Let $\vec u_\alpha$ be the orthonormal basis obtained by applying the Schmidt orthonormal procedure on $\vec v_\alpha$. Then define $\omega_\lambda(y^1,..,y^k)=\mu(0^k)+\sum_{\alpha}y^\alpha\vec u_\alpha$.

		Now set $\eta_\lambda=\exp_{f(p)}^{-1}\circ\exp_{h(p_\lambda)}\circ\omega_\lambda$ and $v_\lambda=\omega_\lambda^{-1}\circ u_\lambda|_{B_{1}(p)}$. By inequality (\ref{C^1closedness}), we have $|\eta_\lambda-\id_{T_{f(p)}N}|_{C^1(B_{10}(0^k))}\le\Psi(\delta|n)$ and $v_\lambda:B_{1}(p)\to\R^k=T_{f(p)}N$ is a $\Psi(\delta|n)$-splitting map with $|v_\lambda-\exp_{f(p)}^{-1}\circ h|\le\Psi(\delta|n)$ which shows (\ref{Jacobian}.2).		
		
		By definition of $f$, for $x=(x^1,...,x^n)$ around $p$, $y=(f^1(x),...,f^k(x))$ is the solution to the equations $$\frac{\partial E}{\partial y^\alpha}(x,y)=0,\quad\alpha=1,..,k.$$ By implicit function theorem, $$\frac{\partial f^\alpha}{\partial x^a}(p)=-K^{\alpha,\beta}\frac{\partial^2 E}{\partial x^a\partial y^\beta}\bigg|_{\substack{x=p\\y=f(p)}},$$ where $K^{\alpha,\beta}$ is the inverse matrix of $\frac{\partial^2 E}{\partial y^\alpha\partial y^\beta}\bigg|_{\substack{x=p\\y=f(p)}}$. Note that by Hessian estimate, $$\left|\frac{\partial^2 d^2_{f_\lambda(x)}}{\partial y^\alpha\partial y^\beta}\bigg|_{\substack{x=p\\y=f(p)}}-2\delta_{\alpha,\beta}\right|\le\Psi(\delta|n),$$ which implies $$\left|\frac{\partial^2 E}{\partial y^\alpha\partial y^\beta}\bigg|_{\substack{x=p\\y=f(p)}}-\delta_{\alpha,\beta}D(p)\right|\le\Psi(\delta|n).$$ Note that the cardinality of $\{\lambda|\phi_\lambda(p)\neq0\}$ is at most $\Lambda(n)$, thus
		\begin{equation}\label{Kestimate}
		\left|K^{\alpha,\beta}-\delta^{\alpha,\beta}D(p)^{-1}\right|\le\Psi(\delta|n).
		\end{equation}

		And we compute,
		
		\begin{equation}\label{Hessian}
		\frac{\partial^2 E}{\partial x^a\partial y^\beta}\bigg|_{\substack{x=p\\y=f(p)}}=\frac 12\sum_\lambda\left(\phi_\lambda(x)\frac{\partial^2d^2(f_\lambda(x),y)}{\partial x^a\partial y^\beta}+\frac{\partial\phi_\lambda}{\partial x^a}\frac{\partial d^2(f_\lambda(x),y)}{\partial y^\beta}\right)\bigg|_{\substack{x=p\\y=f(p)}}.
		\end{equation}

		Observe that to compute the right hand side of (\ref{Hessian}), it suffices to compute those terms with $d(p,p_\lambda)<2.5$. Now we plug $\phi_\lambda=\phi\circ r_\lambda\circ u_\lambda$, $f^\alpha_\lambda=(\eta_\lambda\circ v_\lambda)^\alpha$ and $u_\lambda=\omega_\lambda\circ v_\lambda$ to (\ref{Hessian}) to calculate,
		
		\begin{align*}
		&\frac{\partial^2 E}{\partial x^a\partial y^\beta}\bigg|_{\substack{x=p\\y=f(p)}}
		\\=&\frac 12\sum_\lambda\left(\phi_\lambda(x)\frac{\partial^2d^2(y_2,y_1)}{\partial y_2^\gamma\partial y_1^\beta}\frac{\partial f_\lambda^{\gamma}}{\partial x^a}+\frac{\partial(\phi\circ r_\lambda)}{\partial u_\lambda^\gamma}\frac{\partial u_\lambda^{\gamma}}{\partial x^a}\frac{\partial d^2(f_\lambda(x),y)}{\partial y^\beta}\right)\bigg|_{\substack{x=p\\y=y_1=f(p)\\y_2=f_\lambda(p)}}
		\\=&\frac 12\sum_\lambda\left(\phi_\lambda(x)\frac{\partial^2d^2(y_2,y_1)}{\partial y_2^\gamma\partial y_1^\beta}\frac{\partial (\eta_\lambda\circ v_\lambda)^{\gamma}}{\partial x^a}+\frac{\partial(\phi\circ r_\lambda)}{\partial u_\lambda^\gamma}\frac{\partial (\omega_\lambda\circ v_\lambda)^{\gamma}}{\partial x^a}\frac{\partial d^2(f_\lambda(x),y)}{\partial y^\beta}\right)\bigg|_{\substack{x=p\\y=y_1=f(p)\\y_2=f_\lambda(p)}}
		\\=&\frac 12\sum_\lambda\left(\phi_\lambda(x)\frac{\partial^2d^2(y_2,y_1)}{\partial y_2^\gamma\partial y_1^\beta}\frac{\partial\eta_\lambda^\gamma}{\partial y^\zeta}\circ v_\lambda+\frac{\partial(\phi\circ r_\lambda)}{\partial u_\lambda^\gamma}\frac{\partial d^2(f_\lambda(x),y)}{\partial y^\beta}\frac{\partial\omega_\lambda^\gamma}{\partial y^\zeta}\circ v_\lambda\right)\frac{\partial v_\lambda^{\zeta}}{\partial x^a}\bigg|_{\substack{x=p\\y=y_1=f(p)\\y_2=f_\lambda(p)}}.
		\end{align*}
		
		Put $$A_{\lambda,\beta,\zeta}=\frac 12\left(\phi_\lambda(x)\frac{\partial^2d^2(y_2,y_1)}{\partial y_2^\gamma\partial y_1^\beta}\frac{\partial\eta_\lambda^\gamma}{\partial y^\zeta}\circ v_\lambda+\frac{\partial(\phi\circ r_\lambda)}{\partial u_\lambda^\gamma}\frac{\partial d^2(f_\lambda(x),y)}{\partial y^\beta}\frac{\partial\omega_\lambda^\gamma}{\partial y^\zeta}\circ v_\lambda\right)\bigg|_{\substack{x=p\\y=y_1=f(p)\\y_2=f_\lambda(p)}}.$$
		
		Since $|\sec_N|\le\delta$ and $\inj_N\ge\delta^{-1}$, we have estimates,

		\begin{equation}\label{Estimate}
		\left|\frac{\partial^2d^2(y_2,y_1)}{\partial y_2^\gamma\partial y_1^\beta}\bigg|_{\substack{y_1=f(p)\\y_2=f_\lambda(p)}}+2\delta_{\gamma,\beta}\right|\le\Psi(\delta|n),\quad\left|\frac{\partial d^2(f_\lambda(x),y)}{\partial y^\beta}\bigg|_{\substack{x=p\\y=f(p)}}\right|\le\Psi(\delta|n).
		\end{equation}

		By estimates (\ref{C^1closedness}) and (\ref{Estimate}), we have,
		
		\begin{equation}\label{Aestimate}
		\left|A_{\lambda,\beta,\zeta}+\delta_{\beta,\zeta}\phi_\lambda(p)\right|\le\Psi(\delta|n).
		\end{equation}
		
		Put $C^\alpha_{\lambda,\zeta}=-K^{\alpha,\beta}A_{\lambda,\beta,\zeta}$. Hence $\frac{\partial f^\alpha}{\partial x^a}(p)=\sum_{\lambda}C^{\alpha}_{\lambda,\beta}\frac{\partial v_\lambda^\beta}{\partial x^a}(p)$ which gives (\ref{Jacobian}.1). Combining estimates (\ref{Kestimate}) and (\ref{Aestimate}),
		$$|C^\alpha_{\lambda,\zeta}-\delta^\alpha_{\zeta}\phi_\lambda(p)D(p)^{-1}|\le\Psi(\delta|n)$$ which shows (\ref{Jacobian}.3).
		
	\end{proof}

	\noindent\\[4mm]
	\section{The stability for compact group actions and applications}
	
	\subsection{Proof of Theorem \ref{GroupStatbility}}\quad
	
	As outline of the proof of Theorem \ref{GroupStatbility} in the introduction, we will approximate $\id_M$ by a diffeomorphism $f: (M,g_1)\to (M,g_0)$, constructed in the proof of Theorem \ref{Main} for $N=(M,g_0)$, and thus $d_{g_0}(f,\id_M)\le\Psi(\delta|n)$.  Then we shall apply the standard center of mass construction in averaging $f$ over $G$ to obtain a
	$G$-equivariant map, $F: (M,g_1)\to (M,g_0)$. The verification for non-degeneracy of $\myd F$
	is similar to that of $\myd f$ in the proof of Lemma \ref{Jacobian}.
	
	Up to a scaling, we may assume $M$ is a compact $n$-manifold with two metrics $g_0$ and $g_1$ such that $|\sec_{g_0}|\le\delta$, $\inj_{g_0}\ge\delta^{-1}$, $\Ric_{g_1}\ge -(n-1)\delta$.
	
	Recall that $f:(M,g_1)\to(M,g_0)$ is obtained by gluing splitting maps $\{u_\lambda\}$ via cut-off functions $\{\phi_\lambda\}$ associated with a fixed $1$-net $\{p_\lambda\}$ (see section 3).
	
	Let $\iota_i:G\hookrightarrow\Isom(M,g_i)$, $i=0,1$, be the two isometric actions. Put $f_g\triangleq\iota_0(g^{-1})\circ f\circ\iota_1(g)$. Fixing a bi-invariant probability measure on $G$, define an energy function $\bar E:M\times M\to\R$ by
	$$\bar E(x,y)=\frac 12\int_{G}d_{g_0}(f_g(x),y)^2\myd{g}.$$
	
	For each fixed $x\in (M,g_1)$, there exists a unique $cm(x)$ at which $\bar E(x,\cdot)$ achieves the minimum over $M$. Define a map $F:(M,g_1)\to(M,g_0)$ by $F(x)=cm(x)$. It's not hard to verify that for each $g\in G$, we have $F\circ\iota_1(g)=\iota_0(g)\circ F$. It remains to show $\myd F$ is non-degenerate.

	Similar to Lemma \ref{Jacobian}, we need the following technical lemma.

	\begin{lemma}\label{Jacobian0}
		Given $p\in(M,g_1)$, and a normal coordinate $y^1,...,y^n$ on $T_{F(p)}(M,g_0)$, putting $G_{\lambda}(r)\triangleq\{g\in G|d_{g_1}(p_\lambda,\iota_1(g)p)<r\}$, for each $\lambda$, there exist smooth functions $C^\alpha_{\lambda,\beta}:G\to\R$, $\alpha,\beta=1,...,n,$ and smooth $v_{\lambda}:G_{\lambda}(2.5)\times B_{1}(p,g_1)\to T_{F(p)}(M,g_0)$, such that,
		
		(\ref{Jacobian0}.1) $\myd f_g^\alpha(p)=\sum_{\lambda}C^{\alpha}_{\lambda,\beta}(g)\myd v_{\lambda,g}^\beta(p)$, where $v_{\lambda,g}\triangleq v_{\lambda}(g,\cdot)$.
		
		(\ref{Jacobian0}.2) If $g\notin G_{\lambda}(2.3)$, then $C^\alpha_{\lambda,\beta}(g)=0$.
		
		(\ref{Jacobian0}.3) For each fixed $g\in G_{\lambda}(2.5)$, $v_{\lambda,g}$ is a $(\Psi(\delta|n),n)$-splitting map.

		(\ref{Jacobian0}.4) $|v_{\lambda,g}-\exp_{F(p)}^{-1}|\le\Psi(\delta|n)$, where $\exp_{F(p)}^{-1}$ is respect to $g_0$.
		
		(\ref{Jacobian0}.5) $|C^\alpha_{\lambda,\beta}(g)-\delta^\alpha_\beta\phi_{\lambda,g}(p)D_g^{-1}(p)|\le\Psi(\delta|n)$, where $\phi_{\lambda,g}\triangleq\phi_\lambda\circ\iota_1(g)$ and $D_g\triangleq\sum_\lambda\phi_{\lambda,g}$.

	\end{lemma}
	
	Note that by (\ref{Jacobian0}.2), we don't need to define $v_\lambda$ for those $\lambda$ with $G_\lambda(2.5)=\emptyset$.
	
	\begin{proof}

		For fixed $g\in G$, replace $h$, $\phi_\lambda$, $u_\lambda$, $f_{\lambda}$ and $E$ used in the proof of Lemma \ref{Jacobian} by their $g$-rotated versions, that is, $$h_g\triangleq\iota_0(g^{-1})\circ\iota_1(g),\quad \phi_{\lambda,g}\triangleq\phi_\lambda\circ\iota_1(g),\quad u_{\lambda,g}\triangleq\iota_0(g^{-1})_*\circ u_{\lambda}\circ\iota_1(g),$$$$ f_{\lambda,g}\triangleq\iota_0(g^{-1})\circ f_{\lambda}\circ\iota_1(g),\quad E_g(x,y)\triangleq\frac 12\sum_\lambda\phi_{\lambda,g}(x)d_{g_0}(f_{\lambda,g}(x),y)^2.$$Note that $f_g(x)$ is the unique point at which $E_g(x,\cdot)$ achieves the minimum. Hence we can apply the same construction and calculation as in Lemma \ref{Jacobian} to yield Lemma \ref{Jacobian0}, which we omit the details here.
		
	\end{proof}

	\begin{proof}[Proof of Theorem \ref{GroupStatbility}]

		Fix $p\in (M,g_1)$, a coordinate $x^1,...,x^n$ near $p$, and a normal coordinate $y^1,...,y^n$ on $T_{F(p)}(M,g_0)$. By implicit function theorem,
		
		\begin{equation}\label{JF}
		\frac{\partial F^\alpha}{\partial x^a}(p)=-\frac 12\bar K^{\alpha,\beta}\int_G\frac{\partial^2 d_{g_0}^2(y_2,y_1)}{\partial y_2^\gamma\partial y_1^\beta}\frac{\partial f_g^\gamma}{\partial x^a}\bigg|_{\substack{x=p\\y_1=F(p)\\y_2=f_g(p)}}\myd g,
		\end{equation}
		where $\bar K^{\alpha,\beta}$ is the inverse matrix of $\frac{\partial^2\bar E}{\partial y^\alpha\partial y^\beta}\bigg|_{\substack{x=p\\y=F(p)}}$. We have
		\begin{equation}\label{KEstimate}
		\left|\bar K^{\alpha,\beta}-\delta^{\alpha,\beta}\right|\le\Psi(\delta|n).
		\end{equation}

		By (\ref{Jacobian0}.2), for each fixed $\lambda,\alpha,\beta$, $C_{\lambda,\beta}^\alpha(g) v_{\lambda,g}^\beta(x)$ is a smooth function defined on $G\times B_{1}(p,g_1)$. Hence we can define smooth $v:G\times B_{1}(p,g_1)\to T_{F(p)}(M,g_0)$ by
		
		$$v_g^{\alpha}(x)=\sum_\lambda C_{\lambda,\beta}^\alpha(g) v_{\lambda,g}^\beta(x),$$
		and $v:B_{1}(p)\to T_{F(p)}(M,g_0)$ by	
		$$v^\alpha(x)=-\frac 12\bar K^{\alpha,\beta}\int_G\frac{\partial^2 d_{g_0}^2(y_2,y_1)}{\partial y_2^\gamma\partial y_1^\beta}\bigg|_{\substack{y_1=F(p)\\y_2=f_g(p)}}v_g^\gamma(x)\myd g.$$
		Thus by (\ref{JF}) and (\ref{Jacobian0}.1), $\frac{\partial F^\alpha}{\partial x^a}(p)=\frac{\partial v^\alpha}{\partial x^a}(p)$. What remain is to check the matrix $\frac{\partial v^\alpha}{\partial x^a}(p)$ is non-degenerate.
		
		Combining estimate (\ref{KEstimate}), (\ref{Jacobian0}.4) and (\ref{Jacobian0}.5), for each $x\in B_{1}(p,g_1)$ and those $\lambda,g$ with $g\in G_\lambda(2.5)$, we have
		\begin{equation}\label{7}
		|v_g^\alpha(x)-v_{\lambda,g}^\alpha(x)|\le\Psi(\delta|n),\quad|v^\alpha(x)-v_g^\alpha(x)|\le\Psi(\delta|n).
		\end{equation}
		By (\ref{Jacobian0}.3) and (\ref{7}), $v|_{B_{0.5}(p,g_1)}$ is a $(\Psi(\delta|n),n)$-splitting map, then apply Lemma \ref{NonDegenaracyofSplittingMap} which finishes the proof.
		
	\end{proof}
	
	\subsection{Proof of Corollary \ref{EquivFibration}}\quad
	
	Firstly we need,
	
	\begin{theorem}\cite{MRW}\label{GroupHom}
		Suppose $M_i$ and $N$ are compact $n$-Riemannian manifolds and a $k$-Riemannian manifold ($k\le n$) such that $(M_i,\Gamma_i)\myarrow{GH}(N,\Gamma)$. Then for each large $i$, there exists a Lie group homomorphism $\varphi_i:\Gamma_i\to\Gamma$ which is a $\delta_i$-equivariant GHA, $\delta_i\to0$. Moreover, if $\Ric_{M_i}\ge-(n-1)$ and $k=n$, then $\varphi_i$ is injective.
	\end{theorem}
	
	Note that the original form of Theorem \ref{GroupHom} in \cite{MRW} assumes $\sec_{M_i}\ge-1$. They use this curvature condition to guarantee that $\Gamma$ is a Lie group since their limit may be a singular space. The special case we take here is sufficient for our use.
	
	By the Gromov's compactness criterion and Theorem \ref{GroupHom}, it is clear that under the conditions of Corollary \ref{EquivFibration}, for small $\delta$,
	there is a Lie group homomorphism, $\varphi:G\to H$, such that $(f, \varphi)$ is a $\Psi(\delta|n)$-equivariant GHA, where $f$ is constructed from Theorem \ref{Main}. Then average $\varphi(g^{-1})\circ f\circ g$ over $g\in G$ to obtain an $F$ which will meet our requirement. The proof is the same as that of Theorem \ref{GroupStatbility}, hence we omit the details here. A remark is that in Theorem \ref{GroupStatbility}, $\delta$ does not depend on $G$. While in Corollary \ref{EquivFibration}, the existence of $\varphi:G\to H$ depends on the induced metric on $H$ by the $H$-action on $N$, hence $\delta$ also depends on the $H$-action.

	\subsection{Proof of Corollary \ref{SphereRigidity}}\quad
	
	By the Gromov's compactness criterion, it suffices to consider a sequence of $n$-manifolds $M_i$ satisfying
	$$\Ric_{M_i}\ge (n-1),\quad \vol(\tilde M_i)\ge(1-\delta_i)\vol(S^n),\quad \delta_i\to 0,$$
	and prove that for $i$ large, $M_i$ is diffeomorphic to a space with sectional curvature $1$ by a $\Psi(\delta_i|n)$-GHA.
	
	By diagram (\ref{Equivariantdiagram}), passing to a subsequence, we may assume the following commutative
	equivariant GH-convergence,
	$$
	\xymatrix{
		(\tilde{M}_i,\Gamma_i) \ar[rr]^{GH}\ar[d]_{\pi_i}&&(Y,G) \ar[d]^{\pi_\infty} \\
		M_i\ar[rr]^{GH}&  & X=Y/G,}
	$$
	By quantitative maximal volume rigidity in \cite{Co1}, $Y$ is isometric to the unit sphere
	$S^n_1$. By Theorem \ref{Main} and \ref{GroupHom}, we may assume a diffeomorphism $f_i:\tilde M_i\to S^n_1$ and an injective homomorphism, $\varphi_i: \Gamma_i\to G$ such that $(f_i,\varphi_i)$ is a $\delta_i$-equivariant GHA.
	Applying Theorem \ref{GroupStatbility}, for large $i$, the $\Gamma_i$-action on $S^n_1$ via $f_i$ is conjugate to the $\varphi_i(\Gamma_i)$-action on $S^n_1$ which yields the conclusion.

\bibliographystyle{plain}

\end{document}